\documentclass{amsart}

\usepackage{macros}
\standardmargins
\standardpackages
\standardrenews
\standardlabeling

\draftfalse
\newcommand{\bH}{\mathbb H}
\newcommand{\scrHa}{\mathscr H}
\begin{document}
\title[Rigidity of limit sets for geometrically finite Kleinian groups]{Rigidity of limit sets for nonplanar geometrically finite Kleinian groups of the second kind}

\authorlior\authordavid\authormariusz


\begin{abstract}
We consider the relation between geometrically finite groups and their limit sets in infinite-dimensional hyperbolic space. Specifically, we show that a rigidity theorem of Susskind and Swarup ('92) generalizes to infinite dimensions, while a stronger rigidity theorem of Yang and Jiang ('10) does not.
\end{abstract}
\maketitle

\section{Introduction}

Fix $2\leq d\leq \infty$, let $\bH^d$ denote $d$-dimensional hyperbolic space, and let $\Isom(\bH^d)$ denote the isometry group of $\bH^d$. In this paper we consider the following rigidity question: If $G_1,G_2\leq\Isom(\bH^d)$ are discrete groups whose limit sets $\Lambda(G_1),\Lambda(G_2)$ are equal, are $G_1$ and $G_2$ commensurable? In general the answer is no; additional hypotheses are needed. The following result is due to P. Susskind and G. A. Swarup:

\begin{theorem}[{\cite[Theorem 1]{SusskindSwarup}}; cf. {\cite[Theorem 3]{Greenberg2}} for the case $d = 2$]
\label{theoremsusskindswarup}
Fix $2\leq d < \infty$, and let $G_1,G_2\leq\Isom(\bH^d)$ be discrete groups whose limit sets are equal. If $G_1$ is nonelementary and geometrically finite and is a subgroup of $G_2$, then $G_1$ and $G_2$ are commensurable.
\end{theorem}

The requirement here that $G_1\leq G_2$ is quite a strong hypothesis, and the theorem is certainly false without it. To see this, note that if $G_1,G_2\leq\Isom(\bH^d)$ are lattices, then $\Lambda(G_1) = \del\bH^d = \Lambda(G_2)$, but it is quite possible that $G_1\cap G_2 = \{\id\}$. However, the hypothesis can be replaced by some additional assumptions. Specifically, the following was proven by W.-Y. Yang and Y.-P. Jiang:

\begin{theorem}[{\cite[Corollary 1.2]{YangJiang}}]
\label{theoremequality}
Fix $2\leq d < \infty$, and let $G_1,G_2\leq\Isom(\bH^d)$ be two geometrically finite nonplanar groups of the second kind whose limit sets are equal. Then $G_1$ and $G_2$ are commensurable; in fact,
\[
[\lb G_1,G_2\rb:G_1\cap G_2] < \infty.
\]
\end{theorem}

Here a discrete group $G\leq\Isom(\H^d)$ is said to be \emph{nonplanar} if its limit set is not contained in the closure of any proper totally geodesic subspace of $\bH^d$. We include a proof of Theorem \ref{theoremequality} in Section \ref{sectionequality}, as well as showing that all of its hypotheses are necessary.



In this paper, we show that Theorem \ref{theoremsusskindswarup} can be generalized to infinite dimensions (with ``discrete'' becoming ``strongly discrete'', see below), but Theorem \ref{theoremequality} fails in infinite dimensions. See Theorems \ref{theoremsusskindswarupgeneralization} and \ref{theoremnonrigidity}, respectively.


In Section \ref{sectiondefinitions}, we define the terms used in our theorems and recall some results regarding infinite-dimensional hyperbolic space. In Section \ref{sectionequality}, we prove Theorem \ref{theoremequality}, and in Section \ref{sectioninfdim} we prove our main theorems regarding the infinite dimensional analogues of Theorems \ref{theoremsusskindswarup} and \ref{theoremequality}.

{\bf Acknowledgements.} The first-named author was supported in part by the Simons Foundation grant \#245708. The third-named author was supported in part by the NSF grant DMS-1361677.

\section{Definitions of terms}
\label{sectiondefinitions}
Fix $2\leq d\leq \infty$, and let $\bH^d$ denote $d$-dimensional real hyperbolic space; see \cite[\62]{DSU} for background regarding the case $d = \infty$.
We will use \cite{DSU} as our standard reference regarding Kleinian groups, for the reason that it explicitly considers the infinite-dimensional case.
A group $G\leq\Isom(\bH^d)$ is called \emph{(strongly) discrete} if
\[
\#\{g\in G: \dist(\0,g(\0))\leq R\} < \infty \all R > 0.
\]
The adverb ``strongly'' is used in infinite dimensions since in that case there are other, weaker, notions of discreteness; cf. \cite[\65]{DSU}. The group $G$ is called \emph{nonplanar} if it preserves neither any proper closed totally geodesic subspace of $\bH^d$ nor any point on $\del \bH^d$. This property was called \emph{acting irreducibly} in \cite[\67.6]{DSU}.

The \emph{limit set} of $G$ is the set
\[
\Lambda(G) := \{\xi\in\del\bH^d: \exists (g_n)_1^\infty\text{ in $G$} \;\; g_n(\0)\tendsto n \xi\}.
\]
$G$ is called \emph{nonelementary} if its limit set contains at least three points, in which case its limit set must contain uncountably many points \cite[Proposition 10.5.4]{DSU}.
Recall that a set $A\subset\bH^d$ is said to be \emph{convex} if the geodesic segment connecting any two points of $A$ is contained in $A$, and that when $G$ is nonelementary, the \emph{convex hull} of the limit set is the smallest convex subset of $\bH^d$ whose closure contains $\Lambda$. We denote the convex hull of the limit set by $\CC(G)$.

A strongly discrete group $G\leq\Isom(\bH^d)$ is called \emph{geometrically finite} if there exists a disjoint $G$-invariant collection of horoballs $\scrHa$ and a radius $\sigma > 0$ such that
\[
\CC(G) \subset G(B(\0,\sigma))\cup \bigcup_{H\in\scrHa} H.
\]
This definition appears in the form presented here in \cite[Definition 12.4.1]{DSU}, and in a similar form in \cite[Definition (GF1)]{Bowditch_geometrical_finiteness}.
$G$ is called \emph{convex-cobounded} if the collection $\scrHa$ is empty, i.e. if
\[
\CC(G) \subset G(B(\0,\sigma)).
\]
Finally, $G$ is of \emph{compact type} if its limit set is compact. It was shown in \cite{DSU} that every geometrically finite group is of compact type.

If a sequence $(x_n)_1^\infty$ in $\bH^d$ converges to a point $\xi\in\del\bH^d$, then as usual we call the convergence \emph{radial} if there is a cone with vertex $\xi$ which contains the sequence $(x_n)_1^\infty$. By \cite[Proposition 7.1.1]{DSU}, the convergence is radial if and only if the numerical sequence $(\lb \zero | \xi \rb_{x_n})_1^\infty$ is bounded. Here $\lb \cdot | \cdot \rb$ denotes the Gromov product:
\[
\lb y|\xi\rb_z = \lim_{x\to\xi} \frac12[\dist(z,y) + \dist(z,x) - \dist(y,x)].
\]
Given $\xi\in\Lambda(G)$, we denote by $\busemann_\xi$ the \emph{Busemann function} based at $\xi$, i.e. 
\[
\busemann_\xi(y,z) = \lim_{x\to\xi} [\dist(x,y) - \dist(x,z)].
\]
In the sequel we will find the following results useful:

\begin{proposition}[Minimality of limit sets, {\cite[Proposition 7.4.1]{DSU}}]
\label{propositionminimality}
Fix $G\leq\Isom(\bH^d)$. Any closed $G$-invariant subset of $\del X$ which contains at least two points contains $\Lambda(G)$.
\end{proposition}

\begin{proposition}[{\cite[Proposition 7.6.3]{DSU}}]
Let $G$ be a nonelementary subgroup of $\Isom(\bH^d)$. Then the following are equivalent:
\begin{itemize}
\item[(A)] $G$ is nonplanar.
\item[(B)] There does not exist a nonempty closed totally geodesic subspace $V\propersubset\bH$ whose closure contains $\Lambda(G)$.
\end{itemize}
\end{proposition}

\section{Proof of Theorem \ref{theoremequality}}
\label{sectionequality}
In this section we prove Theorem \ref{theoremequality}, and then show that none of its hypotheses can be dropped. To do so we will need the following theorem:

\begin{theorem}[{\cite[Theorem 2]{Greenberg}}]
\label{theoremgreenberg}
Fix $2\leq d < \infty$, and suppose that $G\leq\Isom(\bH^d)$ is nonplanar and is not dense in $\Isom(\bH^d)$. Then $G$ is discrete.
\end{theorem}

\begin{proof}[Proof of Theorem \ref{theoremequality}]
Fix $2\leq d < \infty$, and let $G_1,G_2\leq\Isom(\bH^d)$ be two geometrically finite nonplanar groups of the second kind whose limit sets are equal. Let $\Lambda$ denote the common limit set of $G_1$ and $G_2$, let $G_+ = \lb G_1,G_2\rb$, and let $G_- = G_1\cap G_2$. Since $\Lambda$ is a $G_+$-invariant closed subset of $\del \bH^d$ which contains at least two points, it follows from Proposition \ref{propositionminimality} that $\Lambda = \Lambda(G_+)$. In particular $\Lambda(G_+)\neq\del\bH^d$, which implies that $G_+$ is not dense in $\Isom(\bH^d)$. On the other hand $G_+$ is nonplanar since it contains a nonplanar subgroup. Thus by Theorem \ref{theoremgreenberg},  $G_+$ is discrete. Applying Theorem \ref{theoremsusskindswarup}, we see that both $G_1$ and $G_2$ are commensurable with $G_+$. Thus $G_1$ and $G_2$ are commensurable, and in particular
\[
[G_+:G_-] \leq [G_+:G_1]\cdot[G_+:G_2] < \infty,
\]
which completes the proof.
\end{proof}

\begin{remark}
\label{remarkhypothesesnecessary}
All three hypotheses of Theorem \ref{theoremequality} are necessary.
\begin{itemize}
\item[1.] The necessity of $G_1$ (and by symmetry $G_2$) being geometrically finite can be seen by letting $G_2$ be a Schottky group generated by two loxodromic isometries $g,h\in\Isom(\bH^2)$ and then letting
\[
G_1 := \lb g^{-n}h g^n:n\in\N\rb.
\]
Clearly $G_1$ and $G_2$ are not commensurable. On the other hand, $G_1$ is a normal subgroup of $G_2$ and so its limit set is preserved by $G_2$; thus by the minimality of limit sets we have $\Lambda(G_1) = \Lambda(G_2)$. Another example based on J\o rgensen fibrations is given at the end of \cite{SusskindSwarup}.
\item[2.] The necessity of $G_1$ (or equivalently, $G_2$) being nonplanar can be seen as follows: Let $G_1$ be a Schottky group generated by two loxodromic isometries $g,h\in\Isom(\bH^4)$ such that
\begin{itemize}
\item[(i)] the axes of $g$ and $h$ are coplanar,
\item[(ii)] the plane $P$ generated by their axes is preserved by $G_1$, and
\item[(iii)] $h$ commutes with every rotation of $\bH^4$ that fixes every point of $P$.
\end{itemize}
Let $j$ be an irrational rotation that fixes every point of $P$, and let
\[
G_2 = \lb g,hj\rb.
\]
Then for all $n\neq 0$, we have $j^n\notin G_2$ and $(hj)^n = h^n j^n\in G_2$ and thus $h^n\notin G_2$. It follows that $G_1$ and $G_2$ are not commensurable. On the other hand, $G_1\vert P = G_2\vert P$, which implies that $\Lambda(G_1) = \Lambda(G_2)$.
\item[3.] The necessity of $G_1$ (or equivalently, $G_2$) being of the second kind can be seen quite easily, as it suffices to consider any two lattices in $\Isom(\bH^d)$ which have no common element.
\end{itemize}
\end{remark}

\section{Infinite dimensions}
\label{sectioninfdim}

In this section we prove our main theorems, namely that while Theorem \ref{theoremsusskindswarup} can be generalized to infinite dimensions, Theorem \ref{theoremequality} cannot. We remark that our counterexample to an infinite-dimensional version of Theorem \ref{theoremequality} is also a counterexample to an infinite-dimensional version of Theorem \ref{theoremgreenberg}, since the proof of Theorem \ref{theoremequality} does not use finite-dimensionality in any way except for the use of Theorem \ref{theoremgreenberg}.

\begin{theorem}
\label{theoremsusskindswarupgeneralization}
Fix $2\leq d\leq\infty$, and let $G_1,G_2\leq\Isom(\bH^d)$ be strongly discrete groups whose limit sets are equal. If $G_1$ is nonelementary and geometrically finite and is a subgroup of $G_2$, then $G_1$ and $G_2$ are commensurable.
\end{theorem}
Note that the finite-dimensional case of this theorem also provides another proof of Theorem \ref{theoremsusskindswarup}.
\begin{proof}[Proof of Theorem \ref{theoremsusskindswarupgeneralization}]
Let $\Lambda$ denote the common limit set of $G_1$ and $G_2$, and let $\CC$ denote the convex hull of $\Lambda$. Fix $\zero\in\CC$ and let $T\subset G_2$ be a transversal\footnote{I.e. a set for which each left coset $g G_1$ of $G_1$ intersects $T$ exactly once.} of $G_2/G_1$ with the following minimality property: for all $g\in T$ and for all $h\in G_1$,
\begin{equation}
\label{minimality}
\dist(\zero,g(\zero)) \leq \dist(\zero,h^{-1} g(\zero)) = \dist(h(\zero),g(\zero)).
\end{equation}
Here $\dist$ denotes the hyperbolic metric on $\bH^d$. Equivalently, (\ref{minimality}) says that $g(\zero)$ is in the closed Dirichlet domain $\DD$ centered at $\zero$ for the group $G_1$ (cf. \cite[Definition 12.1.4]{DSU}).

By contradiction we suppose that $[G_2:G_1] = \#(T) = \infty$. Since $G_1$ is geometrically finite, it is of compact type \cite[Theorem 12.4.4]{DSU}, and thus $G_2$ is also of compact type. On the other hand, $G_2$ is strongly discrete, so by \cite[Proposition 7.7.2]{DSU}, there exists a sequence $(g_n)_1^\infty$ in $T$ so that $g_n(\zero)\to \xi\in \Lambda$. But $G_1$ is geometrically finite, so by \cite[Theorem 12.4.4]{DSU} we have that $\xi$ is either a radial limit point or a bounded parabolic point of $G_1$.

If $\xi$ is a radial limit point of $G_1$, then $\xi$ is also a horospherical limit point of $G_1$, so there exists $h\in G_1$ such that $\busemann_\xi(\zero,h(\zero)) > 0$. But \eqref{minimality} gives
\[
\busemann_\xi(\zero,h(\zero)) = \lim_{n\to\infty} [\dist(\zero,g_n(\zero)) - \dist(h(\zero),g_n(\zero))] \leq 0,
\]
a contradiction.

If $\xi$ is a bounded parabolic point of $G_1$, then $\xi$ is a parabolic point of $G_2$, so by \cite[Remark 12.3.8]{DSU}, $\xi$ is not a radial limit point of $G_2$. We will show that the sequence $(g_n(\zero))_1^\infty$ tends radially to $\xi$, a contradiction.

Given distinct points $p,q\in\bH^d\cup\del\bH^d$, let $\geo pq$ denote the geodesic segment or ray connecting $p$ and $q$. Now, $\CC$ is cobounded in the \emph{quasiconvex core} $\CC_\zero = \bigcup_{g_1,g_2\in G_1} \geo{g_1(\zero)}{g_2(\zero)}$ \cite[Proposition 7.5.3]{DSU}, which is in turn cobounded in the set $A = \bigcup_{g\in G_1} \geo{g(\zero)}\xi$ by the thin triangles condition \cite[Proposition 4.3.1(ii)]{DSU}. Thus, there exists $\sigma > 0$ such that $\CC \subset A^{(\sigma)}$, where $A^{(\sigma)}$ denotes the $\sigma$-thickening of $A$. On the other hand, since $\xi$ is a bounded parabolic point of $G_1$, there exists a $\xi$-bounded set $S\subset\bH^d$ such that $G_1(\zero)\subset H_1(S)$, where $H_1$ is the stabilizer of $\xi$ in $G_1$. Thus, if we let
\[
R = \bigcup_{x\in S} \geo x\xi,
\]
then $\CC \subset \bigcup_{h\in H_1} h(R^{(\sigma)})$.
\begin{claim}
The function
\[
f(y) = \min(\lb \zero|\xi\rb_y, \lb y|\xi\rb_\zero)
\]
is bounded on $R^{(\sigma)}$.
\end{claim}
\begin{subproof}
Fix $y\in R$, say $y\in \geo x\xi$ for some $x\in S$. Since $S$ is $\xi$-bounded, \cite[Proposiiton 4.3.1(i)]{DSU} implies that $\dist(\zero,\geo x\xi)$ is bounded independent of $x$. Let $z\in\geo x\xi$ be the point closest to $\zero$. Then either $\lb y|\xi\rb_z = 0$ or $\lb z|\xi\rb_y = 0$, depending on whether $z$ or $y$ is closer to $\xi$. It follows that $f(y) \leq \dist(\zero,z)$ is bounded independent of $y$. This shows that $f$ is bounded on $R$; since $f$ is uniformly continuous, it is also bounded on $R^{(\sigma)}$.
\end{subproof}
Fix $n\in\N$. Since $x_n := g_n(\zero)\in T\subset \CC\cap\DD$, there exists $h_n\in H_1$ such that $x_n\in h_n(R^{(\sigma)})$. Since $x_n\in\DD$, we have $\dist(\zero,x_n) \leq \dist(\zero,h_n^{-1}(x_n))$ and thus $f(x_n) \leq f(h_n^{-1}(x_n))$. Thus, the function $f$ is bounded on the sequence $(x_n)_1^\infty$. Since $x_n\to\xi$, we must have $\lb x_n | \xi\rb_\zero \to \infty$ (cf. \cite[Observation 3.4.20]{DSU}); thus the sequence $(\lb \zero | \xi \rb_{x_n})_1^\infty$ is bounded. As remarked earlier, this is equivalent to the fact that $x_n\to\xi$ radially, which is a contradiction as observed earlier.
\end{proof}

\begin{theorem}
\label{theoremnonrigidity}
There exist $G_1,G_2\leq\Isom(\bH^\infty)$ convex-cobounded nonplanar groups of the second kind whose limit sets are equal satisfying $G_1\cap G_2 = \{\id\}$. In particular, $G_1$ and $G_2$ are not commensurable.
\end{theorem}

In the proof of Theorem \ref{theoremnonrigidity}, we will make use of the following:

\begin{theorem}[{\cite[Theorem 1.1]{BIM}}]
\label{theoremBIM}
Let $T$ be a tree and let $V\subset T$ denote its set of vertices, and suppose that $\#(V) = \#(\N)$. Then for every $\lambda > 1$, there is an embedding $\Psi_\lambda:V\to\bH^\infty$ and a representation $\pi_\lambda:\Isom(T)\to\Isom(\bH^\infty)$ such that
\begin{itemize}
\item[(i)] $\Psi_\lambda$ is $\pi_\lambda$-equivariant and extends equivariantly to a map $\Psi_\lambda:\del T\to\del\bH^\infty$,
\item[(ii)] for all $x,y\in V$,
\[
\lambda^{\dist(x,y)} = \cosh\dist(\Psi_\lambda(x),\Psi_\lambda(y)), \regtext{ and}
\]
\item[(iii)] the set $\Psi_\lambda(V)$ is cobounded in the convex hull of the set $\Lambda := \Psi_\lambda(\del T)$.
\end{itemize}
\end{theorem}

\begin{proof}[Proof of Theorem \ref{theoremnonrigidity}]
Let $\F_2$ be the free group on two elements, and let $T$ be the right Cayley graph of $\F_2$. Fix any $\lambda > 1$, and apply the previous theorem to get $\Psi_\lambda$, $\pi_\lambda$, and $\Lambda$. Without loss of generality, we can suppose that there is no closed totally geodesic subspace of $\bH^\infty$ containing $\Lambda$; otherwise, replace $\bH^\infty$ by the smallest such subspace.

\begin{lemma}
\label{lemmasharply}
If $\Gamma\leq\Isom(T)$ acts sharply transitively on $V$, then $G := \pi_\lambda(\Gamma)$ is strongly discrete and convex-cobounded; moreover, $\Lambda(G) = \Lambda$.
\end{lemma}
\begin{subproof}
The equation $\Lambda(G) = \Lambda$ follows from the $\pi_\lambda$-equivariance of $\Psi_\lambda$ together with the fact that $\Lambda(\Gamma) = \del T$. Strong discreteness follows from (ii) of Theorem \ref{theoremBIM}, and convex-coboundedness follows from (iii).
\end{subproof}

Let $\Phi:\F_2\to \Isom(T)$ be the natural left action of $\F_2$ on its right Cayley graph, and let $\Gamma_1 = \Phi(\F_2) \leq\Isom(T)$.

\begin{lemma}
There exists $\gamma\in\Isom(T)$ such that $\Gamma_1\cap \gamma^{-1} \Gamma_1 \gamma = \{\id\}$.
\end{lemma}
\begin{subproof}
Write $\F_2 = \lb a,b\rb$, and define $\gamma:\F_2\to\F_2$ by the formula
\[
\gamma(a^{n_1}b^{n_2}\cdots a^{n_{k - 1}}b^{n_k}) =
\begin{cases}
a^{n_1}b^{n_2}\cdots a^{n_{k - 1}}b^{n_k} & \text{ if $n_1\neq 0$}\\
b^{-n_2}\cdots a^{-n_{k - 1}}b^{-n_k} & \text{ if $n_1 = 0$}
\end{cases}.
\]
(The convention here is that $n_i\neq 0$ for $i = 2,\ldots,k - 1$.) It can be verified directly that $\gamma$ preserves edges in the Cayley graph, so $\gamma$ extends uniquely to $\gamma\in\Isom(T)$. By contradiction, suppose there exist $x_1,x_2\in \F_2\butnot\{e\}$ with $\Phi_{x_1} = \gamma^{-1} \Phi_{x_2} \gamma$. Then $\gamma\Phi_{x_1} = \Phi_{x_2}\gamma$; evaluating at $e$ gives $x_2 = \gamma(x_1)$. Write $x = x_1$; we have
\begin{equation}
\label{gxy}
\gamma(xy) = \gamma(x)\gamma(y) \all y\in \F_2.
\end{equation}
Write $x = a^{n_1}b^{n_2}\cdots a^{n_{k - 1}}b^{n_k}$. If $n_1 \neq 0$, then
\[
\gamma(xb) = \gamma(x) b \neq \gamma(x) b^{-1} = \gamma(x) \gamma(b),
\]
and if $n_1 = 0$, then
\[
\gamma(xa) = \gamma(x) a^{-1} \neq \gamma(x) a = \gamma(x) \gamma(a).
\]
Either equation contradicts \eqref{gxy}.
\end{subproof}
Let $\Gamma_2 = \gamma^{-1}\Gamma_1 \gamma$. By Lemma \ref{lemmasharply}, $G_1 = \pi_\lambda(\Gamma_1)$ and $G_2 = \pi_\lambda(\Gamma_2)$ are strongly discrete and convex-cobounded, and $\Lambda(G_1) = \Lambda = \Lambda(G_2)$. On the other hand, $G_1\cap G_2 = \{\id\}$.
\end{proof}

\bibliographystyle{amsplain}

\bibliography{bibliography}

\providecommand{\bysame}{\leavevmode\hbox to3em{\hrulefill}\thinspace}
\providecommand{\MR}{\relax\ifhmode\unskip\space\fi MR }
\providecommand{\MRhref}[2]{%
  \href{http://www.ams.org/mathscinet-getitem?mr=#1}{#2}
}
\providecommand{\href}[2]{#2}
\begin{thebibliography}{1}

\bibitem{Bowditch_geometrical_finiteness}
B.~H. Bowditch, \emph{Geometrical finiteness for hyperbolic groups}, J. Funct.
  Anal. \textbf{113} (1993), no. 2, 245--317.

\bibitem{BIM}
M.~Burger, A.~Iozzi, and N.~Monod, \emph{Equivariant embeddings of trees into
  hyperbolic spaces}, Int. Math. Res. Not. (2005), no. 22, 1331--1369.

\bibitem{DSU}
T.~Das, D.~S. Simmons, and M.~Urba\'nski, \emph{Geometry and dynamics in
  {G}romov hyperbolic metric spaces {I}: with an emphasis on non-proper
  settings}, \url{http://arxiv.org/abs/1409.2155}, preprint 2014.

\bibitem{Greenberg2}
L.~Greenberg, \emph{Discrete groups of motions}, Canad. J. Math. \textbf{12}
  (1960), 415--426.

\bibitem{Greenberg}
\bysame, \emph{Discrete subgroups of the {L}orentz group}, Math. Scand.
  \textbf{10} (1962), 85--107.

\bibitem{SusskindSwarup}
P.~Susskind and G.~A. Swarup, \emph{Limit sets of geometrically finite
  hyperbolic groups}, Amer. J. Math. \textbf{114} (1992), no. 2, 233--250.

\bibitem{YangJiang}
W.-Y. Yang and Y.-P. Jiang, \emph{Limit sets and commensurability of {K}leinian
  groups}, Bull. Aust. Math. Soc. \textbf{82} (2010), no. 1, 1--9.

\end{thebibliography}

\end{document}